\documentclass[12pt]{amsart}
\usepackage[english]{babel}
\usepackage{graphicx}
\usepackage[T1]{fontenc}
\usepackage{color}

\usepackage[colorlinks=true]{hyperref}\hypersetup{urlcolor=blue, citecolor=red}

\makeatletter
\newtheorem*{rep@theorem}{\rep@title}
\newcommand{\newreptheorem}[2]{%
\newenvironment{rep#1}[1]{%
 \def\rep@title{#2 \ref{##1}}%
 \begin{rep@theorem}}%
 {\end{rep@theorem}}}
\makeatother

\newtheorem{Defi}{Definition}
\newtheorem{Thm}{Theorem}
\newtheorem{Prop}[Thm]{Proposition}

\newtheorem{Lem}{Lemma}
\newreptheorem{Lem}{Lemma}
\newtheorem{Rem}{Remark}
\newtheorem*{Prop*}{Proposition}
\newtheorem*{Cor*}{Corollary}
\newtheorem*{Thm*}{Theorem}

\def\eps{\varepsilon}
\def\R{{\mathbb R}}

\def\eps{\varepsilon}

\DeclareMathOperator{\Tr}{Tr}

\newcommand{\na}{\nabla}
\newcommand{\pa}{\partial}

\begin{document}
%%%%%%%%%%%%%%%%%%%%%%%%%%%%%%%%%%%%%%%%%%%%%%%%%%%%%%%%%%%%%%%%%%%%%
\title[Diffusion limit for a Fokker-Planck]{Diffusion limit for a Vlasov-Fokker-Planck Equation\\ in bounded domains}

\bibliographystyle{alpha}

\author[Ludovic Cesbron]{Ludovic Cesbron}
\address{L.C.: DPMMS, Centre for Mathematical Sciences, Wilberforce Road, Cambridge CB3 0WB, United Kingdom.}
\email{lpc31@cam.ac.uk}

\author[Harsha Hutridurga]{Harsha Hutridurga}
\address{H.H.: Department of Mathematics, Imperial College London, London, SW7 2AZ, United Kingdom.}
\email{h.hutridurga-ramaiah@imperial.ac.uk}

\begin{abstract}
We derive a diffusion approximation for the kinetic Vlasov-Fokker-Planck equation in bounded spatial domains with specular reflection type boundary conditions. 
The method of proof involves the construction of a particular class of test functions to be chosen in the weak formulation of the kinetic model. 
This involves the analysis of the underlying Hamiltonian dynamics of the kinetic equation coupled with the reflection laws at the boundary.
This approach only demands the solution family to be weakly compact in some weighted Hilbert space rather than the much tricky $\mathrm L^1$ setting.
\end{abstract}

\maketitle

%\setcounter{tocdepth}{1}
%\tableofcontents

\section{Introduction}
\label{sec:intro}

%%%%%%%%%%%%%%%%%%%%%%%%%%%%%
\subsection{The Vlasov-Fokker-Planck equation}\label{subsec:VFP}
%%%%%%%%%%%%%%%%%%%%%%%%%%%%%

In this paper, we study the diffusion limit of Vlasov-Fokker-Planck equation in a bounded spatial domain with specular reflections on the boundary. The equation we consider models the behavior of a low density gas in the absence of macroscopic force field. Introducing the probability density $f(t,x,v)$, i.e., the probability of finding a particle with velocity $v$ at time $t$ and position $x$, we consider the evolution equation
\begin{subequations}
\begin{align} 
&\partial_t f + v\cdot \nabla_x f = \mathcal{L}f := \nabla_v\cdot \Big( \nabla_v f + vf \Big) & \mbox{ for }(t,x,v)\in(0,T)\times\Omega\times\R^d,\label{eq:vfp}\\
&f(0,x,v) = f^{in}(x,v)  & \mbox{ for }(x,v)\in\Omega\times\R^d. \label{eq:vfpin}
\end{align}
\end{subequations}
The left hand side of \eqref{eq:vfp} models the free transport of particles, while the Fokker-Planck operator $\mathcal{L}$ on the right hand side describes the interaction of the particles with the background. It can be interpreted as a deterministic description of a Langevin equation for the velocity of the particles:
\begin{align*}
\dot{v}(t) = -\nu v(t) + W(t),
\end{align*}
where the friction coefficient $\nu$ will be assumed, without loss of generality, equal to $1$ and $W(t)$ is a Gaussian white noise. We consider (\ref{eq:vfp}) on a smooth bounded domain $\Omega\subset\R^d$ in the sense that there exists a smooth function $\zeta:\R^d \mapsto \R$ such that 
\begin{equation} \label{eq:defzeta}
\Omega
= \lbrace
x\in\R^d
\mbox{ s.t. }
\zeta(x)<0
\rbrace;
\qquad
\partial\Omega
= \lbrace
x\in\R^d
\mbox{ s.t. }
\zeta(x)=0
\rbrace.
\end{equation}
In order to define a normal vector at each point on the boundary we assume that $\nabla_x\zeta(x) \neq 0$ for any $x$ such that $\zeta(x)\ll 1$ and we define the unit outward normal vector, for any $x\in\partial\Omega$, as
\begin{align*}
n(x)
:=
\frac{\nabla_x \zeta(x)}{\left| \nabla_x \zeta(x) \right|}.
\end{align*}
Moreover, we also assume that $\Omega$ is strongly convex, namely that there exists a constant $C_\zeta >0$ such that 
\begin{align}\label{eq:VFP:strong-convex}
\sum_{i,j=1}^d \xi_i \frac{\partial^2 \zeta}{\partial x_i \partial x_j} \xi_j \geq C_\zeta |\xi|^2 \qquad \forall\, \xi \in \R^d.
\end{align}
To define boundary conditions in the phase space, we introduce the following notations:
\begin{align*}
& \Sigma := \lbrace (x,v) \in \partial\Omega\times\R^d \rbrace 
& \mbox{ Phase space Boundary,}
\\
& \Sigma_+ := \lbrace (x,v) \in \partial\Omega\times\R^d\mbox{ such that } v \cdot n(x) > 0 \rbrace
& \mbox{ Outgoing Boundary,}
\\
& \Sigma_- := \lbrace (x,v) \in \partial\Omega\times\R^d\mbox{ such that } v \cdot n(x) < 0 \rbrace
& \mbox{ Incoming Boundary,}
\\
& \Sigma_0 := \lbrace (x,v) \in \partial\Omega\times\R^d\mbox{ such that } v \cdot n(x) = 0 \rbrace
& \mbox{ Grazing set.}
\end{align*}
We denote by $\gamma f$ the trace of $f$ on $\Sigma$. Boundary conditions for \eqref{eq:vfp} take the form of a balance law between the traces of $f$ on $\Sigma_+$ and $\Sigma_-$ which we denote by $\gamma_+ f$ and $\gamma_- f$ respectively. We shall consider, throughout this paper, the specular reflection boundary condition which is illustrated in Figure \ref{fig:SRop} and reads
\begin{align}
\label{eq:VFP:specular-reflection}
\gamma_- f(t,x,v) 
= 
\gamma_+ f(t,x,\mathcal{R}_xv)
\quad 
\mbox{ for }(t,x,v)\in(0,T)\times\Sigma_-,
\end{align}
where $\mathcal{R}_x$ is the reflection operator on the space of velocities given by
\begin{align*}
\mathcal{R}_x v 
:= 
v - 2 \big( v \cdot n(x) \big) n(x).
\end{align*}
Note that this reflection operator changes the direction of the velocity at the boundary but it preserves the magnitude, i.e., $\vert\mathcal{R}_x v\vert = \vert v\vert$.
\begin{figure}[h]
\centering
\includegraphics[width=8cm,height=6.5cm]{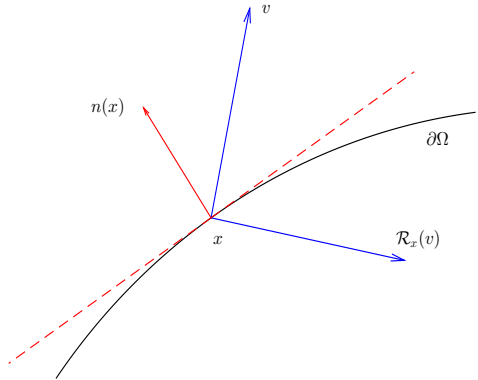}
\caption{Specular reflection operator}
\label{fig:SRop}
\end{figure}

%%%%%%%%%%%%%%%%%%%%%%%%%
\subsection{Main result}\label{ssec:main-result}
%%%%%%%%%%%%%%%%%%%%%%%%%

In order to investigate the diffusion limit of \eqref{eq:vfp}-\eqref{eq:vfpin}, we introduce the Knudsen number $0<\eps\ll 1$ which represents the ratio of the mean free path to the macroscopic length scale, or equivalently the ratio of the mean time between two kinetic interactions to the macroscopic time scale. We rescale time as $t' = \eps t$ and also introduce a coefficient $\eps^{-1}$ in front of the Fokker-Planck operator in \eqref{eq:vfp} to model the number of collision per unit of time going to infinity. The rescaled equation, thus becomes
\begin{subequations} 
\begin{align}
&\eps \partial_t f^\eps + v \cdot \nabla_x f^\eps = \frac{1}{\eps} \nabla_v \cdot \left( v f^\eps + \nabla_v f^\eps \right) & \mbox{for }(t,x,v)\in (0,T)\times\Omega\times\R^d, \label{eq:main:vfp-eps} \\
& f^\eps (0,x,v) = f^{in} (x,v) & \mbox{for }(x,v)\in\Omega\times\R^d, \label{eq:main:vfpin-eps}\\[0.2 cm]
& \gamma_-f^\eps (t,x,v ) = \gamma_+ f^\eps \left(t,x,\mathcal{R}_x v \right) & \mbox{for }(t,x,v)\in (0,T)\times \Sigma_-. \label{eq:main:vfp-specular-eps}
\end{align}
\end{subequations}
In this paper, we investigate the behavior of the solution $f^\eps$ in the $\eps\to0$ limit. The characterization of the asymptotic behavior of $f^\eps(t,x,v)$ is the object of our main result.
\begin{Thm}\label{thm:main:statement} 
Assume the initial datum $f^{in}(x,v)$ satisfies
\begin{align*}
f^{in}(x,v)\geq 0 \hspace{0.5cm} \forall (x,v)\in\Omega\times\R^d;
\qquad
f^{in} \in \mathrm L^2\Big(\Omega\times\R^d, \mathcal{M}^{-1}(v) {\rm d}x{\rm d}v \Big),
\end{align*}
where $\mathcal{M}(v)$ is the centered Gaussian
\begin{align}\label{eq:main:maxwellian-normalization}
\mathcal{M}(v) := \frac{1}{(2\pi)^{d/2}} e^{\frac{-|v|^2}{2}}.
\end{align}
Let $f^\eps(t,x,v)$ be a weak solution to the initial boundary value problem \eqref{eq:main:vfp-eps}-\eqref{eq:main:vfpin-eps}-\eqref{eq:main:vfp-specular-eps}. Then
\begin{align*}
f^\eps(t,x,v) \rightharpoonup \rho(t,x)\mathcal{M}(v)\quad \mbox{ in }\mathrm L^\infty\Big(0,T; \mathrm L^2\Big(\Omega\times\R^d,\mathcal{M}^{-1}(v){\rm d}v{\rm d}x\Big)\Big)\, \mbox{ weak-*}
\end{align*}
as $\eps\to0$, for some $\rho\in \mathrm L^\infty(0,T;\mathrm L^2(\Omega))$.
Furthermore, if the spatial domain $\Omega$ is a ball in $\R^d$ then the limit $\rho(t,x)$ is a weak solution to the diffusion equation
\begin{subequations}
\begin{align}
&\partial_t \rho - \Delta_x \rho = 0 & \mbox{for } (t,x) \in (0,T)\times\Omega,\label{eq:main:diffusion-equation}\\[0.2 cm]
& \rho(0,x) = \rho^{in}(x) & \mbox{for } x\in\Omega,\label{eq:main:diffusion-initial}\\[0.2 cm]
& \nabla_x\rho(t,x) \cdot n(x) = 0 & \mbox{for } (t,x)\in(0,T)\times\partial\Omega,\label{eq:main:diffusion-neumann}
\end{align}
\end{subequations}
with the initial datum
\begin{align*}
\rho^{in}(x) = \int\limits_{\R^d} f^{in}(x,v)\, {\rm d}v.
\end{align*}
\end{Thm}

%%%%%%%%%%%%%%%%
\subsection{Plan of the paper}
%%%%%%%%%%%%%%%%

Section \ref{sec:strategy-proof} gives some heuristics with regard to the strategy of proof for Theorem \ref{thm:main:statement}. In particular, we compare our method of proof with some standard techniques used to prove the diffusion limit for the kinetic Fokker-Planck equation.
In section \ref{sec:soln-VFP}, we define an appropriate notion of weak solution to our initial boundary value problem. In section \ref{sec:auxiliary-problem}, we develop the theory of constructing a special class of test functions using an auxiliary problem. Finally, in section \ref{sec:macro}, we arrive at the parabolic limit equation, thus proving Theorem \ref{thm:main:statement}. In Appendix, we give regularity results associated with some Hamiltonian dynamics that we will need to study the aforementioned auxiliary problem.\\

\noindent{\bf Acknowledgments.} 
The authors would like to acknowledge the support of the ERC grant \textsc{matkit}. 
H.H. also acknowledges the support of the EPSRC programme grant ``Mathematical fundamentals of Metamaterials for multiscale Physics and Mechanic'' (EP/L024926/1).
The authors would like to thank Antoine Mellet for his fruitful suggestions on a preliminary version of this manuscript.
The authors also extend their humble thanks to Cl\'ement Mouhot for interesting discussions on this problem.

%%%%%%%%%%%%%%%%%%%%%%%%%%%%%
\section{Strategy of the proof}\label{sec:strategy-proof}
%%%%%%%%%%%%%%%%%%%%%%%%%%%%%

In this section, we lay out the strategy of our proof for Theorem \ref{thm:main:statement}. We would like to demonstrate the novelty in our approach by citing some comparisons with the standard techniques used in the diffusion approximation for the Vlasov-Fokker-Planck equation. Those techniques were introduced in 1987 by Degond and Mas-Gallic \cite{Degond_1987} for the one dimensional case in bounded domains. They were later improved, in 2000, by Poupaud and Soler \cite{Poupaud_2000} where they consider the more complicated Vlasov-Poisson-Fokker-Planck equation on the whole space and established the diffusion limit for a small enough time interval. More recently, improving the result further, Goudon \cite{Goudon_2005} established in 2005 the global-in-time convergence in dimension 2 with bounds on the entropy and energy of the initial data so as to ensure that singularities do not develop in the limit system and finally, in 2010, El Ghani and Masmoudi proved in \cite{El_Ghani_2010} the global-in-time convergence in higher dimensions with similar initial bounds.\\
In these papers, the analysis with regard to this nonlinear model is quite involved. Let us simply present the analysis in \cite{Poupaud_2000} adapted to the linear model \eqref{eq:main:vfp-eps}. The idea is to consider the continuity equation for the local densities $\rho^\eps(t,x)$ given by
\begin{align*}
\frac{\partial \rho^\eps}{\partial t}(t,x)
+ \frac{1}{\eps} \nabla_x\cdot j^\eps(t,x)
= 0,
\end{align*}
where the current density $j^\eps(t,x)$ is defined as
\begin{align*}
j^\eps(t,x)
:=
\int\limits_{\R^d}
vf^\eps(t,x,v)
\, {\rm d}v.
\end{align*}
The principal idea is to obtain
\begin{equation}\label{eq:strategy:poupaud-soler-limits}
\begin{aligned}
& \rho^\eps \rightharpoonup \rho \qquad \mbox{ weakly in }\mathrm L^1((0,T)\times\Omega),\\
& \frac{1}{\eps}j^\eps \rightharpoonup \nabla_x \rho \qquad \mbox{ in }\mathcal{D}'((0,T)\times\Omega)
\end{aligned}
\end{equation}
as $\eps\to0$. The article \cite{Poupaud_2000} is concerned with the analysis in the full spatial domain $\R^d$. In order to derive the limit boundary condition -- we refer the interested reader to the paper \cite{Wu_2015} of Wu, Lin and Liu for more details -- one can multiply the specular reflection boundary condition \eqref{eq:main:vfp-specular-eps} by $\left( v\cdot n(x)\right)$ and integrate over the incoming velocities at the point $x\in\partial\Omega$ yielding
\begin{align*}
\int\limits_{v\cdot n(x)<0}
\gamma_- f^\eps(t,x,v) 
\left(
v\cdot n(x)
\right)
\, {\rm d}v
= 
\int\limits_{v\cdot n(x)<0}
\gamma_+f^\eps(t,x,\mathcal{R}_x(v))
\left(
v\cdot n(x)
\right)
\, {\rm d}v.
\end{align*}
Making the change of variables $w=\mathcal{R}_x(v)$ on the right hand side of the above expression yields
\begin{align*}
\int\limits_{v\cdot n(x)<0}
\gamma_- f^\eps(t,x,v) 
\left(
v\cdot n(x)
\right)
\, {\rm d}v
= 
-
\int\limits_{w\cdot n(x)>0}
\gamma_+f^\eps(t,x,w)
\left(
w\cdot n(x)
\right)
\, {\rm d}w.
\end{align*}
This implies the following
\begin{align*}
\int\limits_{\R^d}
\gamma f^\eps 
\left(
v\cdot n(x)
\right)
\, {\rm d}v
= 0
\implies
j^\eps(t,x)\cdot n(x) = 0. 
\end{align*}
Taking the limits \eqref{eq:strategy:poupaud-soler-limits} into consideration, we do have the homogeneous Neumann condition on the boundary in the $\eps\to0$ limit.

Our strategy is essentially different in the sense that we exploit the hyperbolic structure of the Vlasov-Fokker-Planck equation that appears in Fourier space, as we will explain in section \ref{sec:auxiliary-problem}, and which reveals, when coupled with the reflective boundaries, the underlying Hamiltonian dynamics of the kinetic equation. We will take advantage of the dynamics by constructing a special class of test functions for the weak formulation \eqref{eq:soln-VFP:weak-soln-definition} of the initial boundary value problem \eqref{eq:main:vfp-eps}-\eqref{eq:main:vfp-specular-eps}-\eqref{eq:main:vfpin-eps} and then passing to the limit for such test functions, only using the weak $\mathrm L^2$-compactness result (see Proposition \ref{prop:soln-VFP:weak-limit}).

\subsection{Efficiency of our approach}
To justify the interest of our method, we prove the diffusion limit in full space, i.e., when $\Omega=\R^d$. It only takes a few lines which shows how efficient our method is in the Fokker-Planck context. Let us consider the scaled (diffusive scaling) Vlasov-Fokker-Planck equation for the probability density $f^\eps(t,x,v)$ in the full space.
\begin{align*} 
&\eps\partial_t f^\eps + v\cdot \nabla_x f^\eps = \frac{1}{\eps}\nabla_v\cdot \Big( \nabla_v f^\eps + vf^\eps \Big) & \mbox{ for }(t,x,v)\in(0,T)\times\R^d\times\R^d,\\
&f^\eps(0,x,v) = f^{in}(x,v)  & \mbox{ for }(x,v)\in\R^d\times\R^d.
\end{align*}
This equation has a unique weak solution $f^\eps(t,x,v)$ which satisfies
\begin{align*}
f^\eps \in \mathrm L^2\big( (0,T)\times\R^d_x ; \mathrm H^1(\R^d_v)\big) \mbox{ and } \partial_t f^\eps + v\cdot\nabla_x f^\eps \in \mathrm L^2\big( (0,T)\times\R^d_x ; \mathrm H^{-1}(\R^d_v)\big) 
\end{align*}
as was proven by Degond in the appendix of \cite{Degond_1986}. Moreover, the Fokker-Planck operator is dissipative in the sense that 
\begin{align*}
- \iint\limits_{\R^d\times\R^d} f^\eps \mathcal{L}\big(f^\eps\big) \frac{{\rm d}x {\rm d}v}{\mathcal{M}(v)} \geq 0
\end{align*}
from which, as we will prove in section \ref{sec:soln-VFP}, we can show that $f^\eps$ converges weakly$\ast$ in $\mathrm L^\infty\left( 0,T; \mathrm L^2\left(\Omega\times\R^d, \mathcal{M}^{-1}(v){\rm d}x{\rm d}v\right) \right)$ to $\rho(t,x)\mathcal{M}(v)$ where $\rho(t,x)$ is the limit of the local densities $\rho^\eps (t,x) := \int_{\R^d} f^\eps\, {\rm d}v$.\\

For any $\psi \in \mathcal{C}^\infty_c ((0,T)\times\R^d)$ we construct the test function $\phi^\eps(t,x,v) = \varphi(t,x+\eps v)$ with which the weak formulation of the Vlasov-Fokker-Planck equation reads
\begin{align*}
\iiint\limits_{(0,T)\times\R^d\times\R^d}
f^\eps(t,x,v) 
\Big(
\eps^2\partial_t \phi^\eps 
+
\eps v\cdot \nabla_x \phi^\eps 
& +
\Delta_v \phi^\eps 
-
v\cdot \nabla_v \phi^\eps 
\Big)
(t,x,v)
\, {\rm d}v\, {\rm d}x\, {\rm d}t
\\
& +
\eps^2
\iint\limits_{\R^d\times\R^d}
f^{in}(x,v)
\phi^\eps(0,x,v)
\, {\rm d}v\, {\rm d}x
= 0.
\end{align*}
Our particular choice of the test functions enables us to have
\begin{align*}
\eps v\cdot \nabla_x \phi^\eps 
=
v\cdot \nabla_v \phi^\eps
\qquad\mbox{ and }\qquad
\Delta_v \phi^\eps = \eps^2 \Delta_x \phi^\eps.
\end{align*}
Thus, we have
\begin{align*}
\iiint\limits_{(0,T)\times\R^d\times\R^d}
f^\eps(t,x,v) 
\Big(
\partial_t \varphi
+
\Delta_x \varphi 
\Big)
(t,x+\eps v)
\, {\rm d}v\, {\rm d}x\, {\rm d}t
+
\iint\limits_{\R^d\times\R^d}
f^{in}(x,v)
\varphi(0,x+\eps v)
\, {\rm d}v\, {\rm d}x
= 0.
\end{align*}
Passing to the limit in the above expression as $\eps\to0$, using the weak convergence of $f^\eps$ and the regularity of $\varphi$ with respect to both its variables, yields
\begin{align*}
\iint\limits_{(0,T)\times\R^d}
\rho(t,x) 
\Big(
\partial_t \varphi
+
\Delta_x \varphi 
\Big)
(t,x)
\, {\rm d}x\, {\rm d}t
+
\int\limits_{\R^d}
\rho^{in}(x)
\varphi(0,x)
\, {\rm d}x
= 0
\end{align*}
which is the weak formulation of
\begin{align*}
&\partial_t \rho - \Delta_x \rho  = 0 \qquad & \mbox{ for }(t,x)\in (0,T)\times\R^d,
\\
&\rho(0,x)  = \rho^{in}(x) & \mbox{ for }x\in \R^d.
\end{align*}

%%%%%%%%%%%%%%%%%%%%%%%%%%%%%%%
\section{Solutions of the Vlasov-Fokker-Planck equation}\label{sec:soln-VFP}
%%%%%%%%%%%%%%%%%%%%%%%%%%%%%%%

Several works from the 80's and 90's investigate the existence of solution to the Vlasov-Fokker-Planck equation. We refer the interested reader to \cite{Degond_1986} for the global existence of smooth solution in the whole space in space dimensions $1$ and $2$ and to \cite{Carrillo_1998} for global weak solutions on a bounded domain with absorbing-type boundary condition. More recently, Mellet and Vasseur established existence of global weak solution with reflection-law on the boundary in \cite{Mellet_2007}.

%%%%%%%%%%%%%%%%%%%%%%%%%%%%%%%%
\subsection{Existence of weak solution}\label{subsec:notion-weak-solution}
%%%%%%%%%%%%%%%%%%%%%%%%%%%%%%%%
The present work is in a very similar framework and we will therefore use the same kind of definition for weak solution as in \cite{Mellet_2007}.
\begin{Defi} \label{def:soln-VFP:weak-solution}
We say that $f(t,x,v)$ is a weak solution of \eqref{eq:vfp}-\eqref{eq:vfpin}-\eqref{eq:VFP:specular-reflection} on $[0,T]$ if
\begin{equation}\label{eq:soln-VFP:regularity-f}
\begin{aligned}
&f(t,x,v) \geq 0 \qquad \forall (t,x,v)\in [0,T]\times\Omega\times\R^d,\\[0.2 cm]
& f\in C\left( [0,T];\mathrm L^1(\Omega\times\R^d)\right) \cap \mathrm L^\infty \left(0,T; \mathrm L^1\cap \mathrm L^\infty(\Omega\times\R^d)\right) 
\end{aligned}
\end{equation}
and \eqref{eq:vfp} holds in the sense that for any $\phi(t,x,v)$ such that
\begin{equation} \label{eq:soln-VFP:weak-soln-test-function}
\begin{aligned}
&\phi \in C^\infty \left( [0,T]\times\overline{\Omega}\times\R^d\right), \qquad \phi(T,\cdot,\cdot)=0, \\[0.2 cm]
&\gamma_+\phi(t,x,v) = \gamma_-\phi\left(t,x,\mathcal{R}_x(v)\right) \quad \forall (t,x,v)\in [0,T]\times\Sigma_+,
\end{aligned}
\end{equation}
we have
\begin{equation}\label{eq:soln-VFP:weak-soln-definition}
\begin{aligned}
\iiint\limits_{(0,T)\times\Omega\times\R^d} f(t,x,v)
\Big( 
\partial_t\phi + v\cdot \nabla_x \phi - & v\cdot \nabla_v \phi + \Delta_v \phi 
\Big) (t,x,v)
\, {\rm d}v\, {\rm d}x\, {\rm d}t
\\
&+ \iint\limits_{\Omega\times\R^d} f^{in}(x,v)\phi(0,x,v)
\, {\rm d}v\, {\rm d}x
= 0.
\end{aligned}
\end{equation}
\end{Defi}    

Such a definition is required as it is well-known for kinetic equations that the specular reflection condition causes a loss in regularity of the solution, in comparison with absorption type boundary condition, as is explained in detail in \cite{Mischler_2010}. Hence, we introduce the above formulation where the boundary condition is satisfied in a weak sense. With such a notion of weak solution, we have the following result of existence from \cite{Mellet_2007}.

\begin{Thm}\label{thm:soln-VFP:existence}
Let the initial data $f^{in}(x,v)$ satisfy
\begin{align}\label{eq:soln-VFP:regularity-initial-1}
&f^{in}(x,v)\geq 0 \quad \forall (x,v)\in\Omega\times\R^d;\qquad
f^{in} \in \mathrm L^2(\Omega\times\R^d, \mathcal{M}^{-1}(v) {\rm d}x{\rm d}v).
\end{align}
Then there exists a weak solution to \eqref{eq:vfp}-\eqref{eq:vfpin} satisfying \eqref{eq:VFP:specular-reflection} defined globally-in-time. Moreover, we have the a priori estimate
\begin{align}\label{eq:soln-VFP:apriori}
\underset{t\in(0,T)}{\sup} \iint\limits_{\Omega\times\R^d} |f(t,x,v)|^2 \frac{{\rm d}x{\rm d}v}{\mathcal{M}(v)} + \int\limits_0^T \mathcal{D}(f)(t)\, {\rm d}t \leq \iint\limits_{\Omega\times\R^d} |f^{in}(x,v)|^2 \frac{{\rm d}x{\rm d}v}{\mathcal{M}(v)},
\end{align}
where the dissipation $\mathcal{D}$ is given by:
\begin{align} \label{eq:soln-VFP:dissipation}
\mathcal{D}(f) = -2\iint\limits_{\Omega\times\R^d} f(t,x,v) \mathcal{L}f(t,x,v) \frac{{\rm d}x{\rm d}v}{\mathcal{M}(v)}.
\end{align}
\end{Thm}

The proof of the above theorem is similar to the proof of Theorem 2.2 in \cite{Mellet_2007}. It consists of approximating the specular reflection condition \eqref{eq:VFP:specular-reflection} through induction on Dirichlet boundary conditions and showing that regularity \eqref{eq:soln-VFP:regularity-f} and estimates \eqref{eq:soln-VFP:apriori} hold as we pass to the limit in the induction procedure. As it is not the principal focus of this article, we will not give a detailed proof of Theorem \ref{thm:soln-VFP:existence}. However, in an effort to motivate the estimate \eqref{eq:soln-VFP:apriori}, we present the following, rather formal, computation.\\

Assume $f$ has a trace in $\mathrm L^2(0,T;\mathrm L^2(\Sigma_+))$. Multiply \eqref{eq:vfp} by $\mathcal{M}^{-1}(v)f(t,x,v)$ and integrate over the phase space $\Omega\times\R^d$ yielding
\begin{equation}\label{eq:soln-VFP:energy-estimate-1}
\begin{aligned}
&
\frac{{\rm d}}{{\rm d}t}
\iint\limits_{\Omega\times\R^d}
|f(t,x,v)|^2
\, \frac{{\rm d}v\, {\rm d}x}{\mathcal{M}(v)}
+ 
\iint\limits_{\Omega\times\R^d}
v \cdot \nabla_x \left( f(t,x,v) \right)^2
\, \frac{{\rm d}v\, {\rm d}x}{\mathcal{M}(v)}
=
2\iint\limits_{\Omega\times\R^d}
\mathcal{L}f(t,x,v)
f(t,x,v)
\, \frac{{\rm d}v\, {\rm d}x}{\mathcal{M}(v)}.
\end{aligned}
\end{equation}
For the second term on the left hand side of the above expression, using the assumption on the trace of $f$, we write
\begin{align*}
&\iint\limits_{\Omega\times\R^d} v \cdot \nabla_x \left( f(t,x,v) \right)^2
\, \frac{{\rm d}v\, {\rm d}x}{\mathcal{M}(v)}
= 
\iint\limits_\Sigma 
\left|\gamma f\right|^2
\left(v\cdot n(x)\right)
\, \frac{{\rm d}v\, {\rm d}x}{\mathcal{M}(v)}
\\
&
= 
\iint\limits_{\Sigma_+} 
\left|\gamma_+ f(t,x,v) \right|^2 
|v\cdot n(x)| 
\, \frac{{\rm d}v\, {\rm d}\sigma(x)}{\mathcal{M}(v)}
- 
\iint\limits_{\Sigma_-} 
\left|\gamma_- f(t,x,v) \right|^2
|v\cdot n(x)|
\, \frac{{\rm d}v\, {\rm d}\sigma(x)}{\mathcal{M}(v)}
\end{align*}
where, using the specular reflection \eqref{eq:VFP:specular-reflection} and the fact that $\mathcal{M}(v)$ is radial, the change of variable $w=\mathcal{R}_x(v)$ yields
\begin{align*}
\iint\limits_{\Sigma_-}
\left|\gamma_- f(t,x,v) \right|^2
|v\cdot n(x)|
\, \frac{{\rm d}v\, {\rm d}\sigma(x)}{\mathcal{M}(v)}
=
\iint\limits_{\Sigma_+} 
\left|\gamma_+ f(t,x,w) \right|^2
|w\cdot n(x)|
\, \frac{{\rm d}w\, {\rm d}\sigma(x)}{\mathcal{M}(w)}.
\end{align*}
This implies that the second term on the left hand side of the expression \eqref{eq:soln-VFP:energy-estimate-1} does not contribute. Hence, we arrive at the following identity
\begin{align*}
\frac{{\rm d}}{{\rm d}t}
\iint\limits_{\Omega\times\R^d}
|f(t,x,v)|^2
\, \frac{{\rm d}v\, {\rm d}x}{\mathcal{M}(v)}
= - \mathcal{D}(f).
\end{align*}
Integrating the above identity over the time interval $(0,T)$ yields the a priori estimate \eqref{eq:soln-VFP:apriori}.

%%%%%%%%%%%%%%%%%%%%%%%%%%%%%%%%%
\subsection{Uniform a priori estimate}\label{ssec:uniform-a-priori-estimates}
%%%%%%%%%%%%%%%%%%%%%%%%%%%%%%%%%

The notion of weak solution (Definition \ref{def:soln-VFP:weak-solution}) and the theorem of existence (Theorem \ref{thm:soln-VFP:existence}) hold for the scaled equation \eqref{eq:main:vfp-eps}-\eqref{eq:main:vfpin-eps}-\eqref{eq:main:vfp-specular-eps} for any $\eps>0$. The scaling only changes the estimate \eqref{eq:soln-VFP:apriori} which becomes
\begin{align} \label{eq:soln-VFP:uniform-apriori}
\underset{t\in(0,T)}{\sup} \iint\limits_{\Omega\times\R^d}
|f^\eps (t,x,v)|^2
\, \frac{{\rm d}v\, {\rm d}x}{\mathcal{M}(v)}
+
\frac{1}{\eps^2}
\int\limits_0^T
\mathcal{D}(f^\eps)(t)
\, {\rm d}t
\leq 
\iint\limits_{\Omega\times\R^d}
|f^{in}(x,v)|^2
\, \frac{{\rm d}v\, {\rm d}x}{\mathcal{M}(v)}
\end{align}
as one can formally see by doing the computation involving \eqref{eq:soln-VFP:energy-estimate-1} with the scaling. We shall use the estimate \eqref{eq:soln-VFP:uniform-apriori} to prove the following result.

\begin{Prop}\label{prop:soln-VFP:weak-limit}
Let $f^\eps(t,x,v)$ be a weak solution of the scaled Vlasov-Fokker-Planck equation with specular reflection \eqref{eq:main:vfp-eps}-\eqref{eq:main:vfpin-eps}-\eqref{eq:main:vfp-specular-eps} in the sense of Definition \ref{def:soln-VFP:weak-solution} with an initial datum $f^{in}(x,v)$ which satisfies \eqref{eq:soln-VFP:regularity-initial-1}. Then there exists $\rho \in L^2((0,T)\times\Omega)$ such that
\begin{align}\label{eq:soln-VFP:weak-limit-f-eps-rho-M}
f^\eps \rightharpoonup \rho(t,x)\mathcal{M}(v) \quad \mbox{ weakly in } \mathrm L^2\left( 0,T;\mathrm L^2\left(\Omega\times\R^d, \mathcal{M}^{-1}(v){\rm d}x{\rm d}v\right) \right)
\end{align}
where $\rho(t,x)$ is the weak-* limit of the local densities
\begin{align}\label{eq:soln-VFP:local-density}
\rho^\eps(t,x) 
:=
\int\limits_{\R^d} f^\eps(t,x,v)
\, {\rm d }v
\end{align}
in $\mathrm L^\infty(0,T; \mathrm L^2(\Omega))$.
\end{Prop}

\begin{proof}
The proof relies on the properties of the dissipation \eqref{eq:soln-VFP:dissipation} in the estimate \eqref{eq:soln-VFP:uniform-apriori}. Remark that the Fokker-Planck operator can be rewritten as
\begin{align*}
\mathcal{L}f^\eps 
= 
\nabla_v \cdot
\left(
\mathcal{M}(v)
\nabla_v
\left(
\frac{f^\eps}{\mathcal{M}(v)}
\right)
\right).
\end{align*}
This helps us deduce that the dissipation $\mathcal{D}$ is positive semi-definite, i.e.,
\begin{align*}
\mathcal{D}(f^\eps) 
= 
- \iint\limits_{\Omega\times\R^d}
\frac{f^\eps}{\mathcal{M}(v)}
\mathcal{L}f^\eps 
\, {\rm d}v\, {\rm d}x
= 
\iint\limits_{\Omega\times\R^d}
\left|\nabla_v \left( \frac{f^\eps}{\mathcal{M}(v)} \right) \right|^2 \mathcal{M}(v)
\, {\rm d}v\, {\rm d}x
\geq 
0.
\end{align*}
The non-negativity of $\mathcal{D}$ in \eqref{eq:soln-VFP:uniform-apriori} yields the following uniform (with respect to $\eps$) bound.
\begin{align}\label{eq:soln-VFP:uniform-bound-L-infty-L2}
\|f^\eps\|_{\mathrm L^\infty(0,T;\mathrm L^2(\Omega\times\R^d,\mathcal{M}^{-1}(v){\rm d}x{\rm d}v))}
\le C.
\end{align}
Hence, we can extract a sub-sequence and there exists a limit $\overline{f}(t,x,v)$ such that
\begin{align*}
f^\eps \rightharpoonup \overline{f}
\quad \mbox{ in } \mathrm L^\infty\left( 0,T; \mathrm L^2\left(\Omega\times\R^d, \mathcal{M}^{-1}(v){\rm d}x{\rm d}v\right) \right)\, \mbox{ weak-*}
\end{align*}
as $\eps\to0$. Moreover, using the Cauchy-Schwarz inequality, we have
\begin{equation*}
|\rho^\eps(t,x)| 
= \left|
\int\limits_{\R^d}
\frac{f^\eps}{\mathcal{M}^{1/2}}\mathcal{M}^{1/2}
\, {\rm d}v
\right|
\le 
\left(
\int\limits_{\R^d}
|f^\eps|^2
\frac{{\rm d}v}{\mathcal{M}(v)}
\right)^{\frac{1}{2}} 
\left(
\int\limits_{\R^d}
\mathcal{M}(v)
\, {\rm d}v 
\right)^{\frac{1}{2}}.
\end{equation*}
Since $\mathcal{M}(v)$ is normalized \eqref{eq:main:maxwellian-normalization}, integrating the above inequality in the spatial variable and taking supremum over the time interval $[0,T]$ yields the following estimate
\begin{align}\label{eq:soln-VFP:uniform-bound-rho-L-infty-L2}
\|\rho^\eps\|_{\mathrm L^\infty(0,T;\mathrm L^2(\Omega))}
\le C,
\end{align}
where we have used the estimate \eqref{eq:soln-VFP:uniform-bound-L-infty-L2}. Again, we can extract a sub-sequence and there exists a limit $\rho(t,x)$ such that
\begin{align*}
\rho^\eps \rightharpoonup \rho
\quad \mbox{ in }\mathrm L^\infty\left( 0,T;\mathrm L^2(\Omega) \right)\, \mbox{ weak-*}.
\end{align*}
Remark that the dissipation can be successively written as
\begin{align*}
\mathcal{D}(f^\eps) 
= 
\iint\limits_{\Omega\times\R^d}
\left|
\nabla_v
\left(
\frac{f^\eps}{\mathcal{M}(v)}
\right)
\right|^2
\mathcal{M}(v)
\, {\rm d}v\, {\rm d}x
=
\iint\limits_{\Omega\times\R^d}
\left|
\nabla_v 
\left( 
\frac{f^\eps-\rho^\eps \mathcal{M}(v)}{\mathcal{M}(v)}
\right)
\right|^2
\mathcal{M}(v)
\, {\rm d}v\, {\rm d}x.
\end{align*}
Using Poincar\'e inequality for the Gaussian measure in the velocity variable yields the existence of a constant $\theta>0$ such that
\begin{align*}
\mathcal{D}(f^\eps)
\geq 
\theta 
\iint\limits_{\Omega\times\R^d}
\left|
\frac{f^\eps-\rho^\eps \mathcal{M}(v)}{\mathcal{M}(v)}
\right|^2 
\mathcal{M}(v)
\, {\rm d}v\, {\rm d}x
=
\theta
\iint\limits_{\Omega\times\R^d}
\left|
f^\eps - \rho^\eps \mathcal{M}(v) 
\right|^2 
\frac{{\rm d}v{\rm d}x}{\mathcal{M}(v)}.
\end{align*}
Since \eqref{eq:soln-VFP:uniform-apriori} implies that the dissipation tends to zero as $\eps$ tends to zero, we have
\begin{align*}
f^\eps-\rho^\eps \mathcal{M}(v) \to 0
\quad
\mbox{ strongly in }
\mathrm L^2(0,T;\mathrm L^2(\Omega\times\R^d,\mathcal{M}^{-1}(v){\rm d}x{\rm d}v)).
\end{align*}
This concludes the proof.
\end{proof}

%%%%%%%%%%%%%%%%%%%%%%%%
\section{Auxiliary problem}\label{sec:auxiliary-problem}
%%%%%%%%%%%%%%%%%%%%%%%%

The auxiliary problem that we consider is inspired by the hyperbolic structure of the Vlasov-Fokker-Planck equation in Fourier space. Indeed, if we consider \eqref{eq:main:vfp-eps} in the whole space and apply Fourier transform in $x$ and $v$ variables (with respective Fourier variables $p$ and $q$), we have
\begin{align*}
\eps\partial_t \widehat{f^\eps}
+ 
\left(
p - \frac{1}{\eps} q
\right)
\cdot
\nabla_q
\widehat{f^\eps}
=
\frac{1}{\eps}
|q|^2
\widehat{f^\eps}
\end{align*}
which is a hyperbolic equation, its characteristic lines given by $ ( p- \eps^{-1} q)\cdot \nabla_q$. The motivation behind the auxiliary problem is to choose a test function which will be constant along those lines (translated in an adequate way to the non-Fourier space) and satisfy the specular reflection condition \eqref{eq:VFP:specular-reflection}. This auxiliary problem was first introduced in \cite{Cesbron_2012} in the whole space and then improved in \cite{Cesbron_2016} to handle bounded domains and in particular specular reflection boundary conditions. For the sake of completeness, let us present the construction of a solution to this problem in a strongly convex domain with specular reflections on the boundary.

%%%%%%%%%%%%%%%%%%%%%%%%%%%
\subsection{Geodesic Billiards and Specular cycles}\label{subsec:billiards}
%%%%%%%%%%%%%%%%%%%%%%%%%%%

For any $\psi \in \mathcal{C}^\infty (\overline{\Omega})$ we construct $\varphi(x,v)$ through the following boundary value problem.
\begin{equation}\label{eq:Aux:auxiliary-problem}
\left\{
\begin{aligned}
&v \cdot \nabla_x \varphi 
- v \cdot \nabla_v \varphi
 = 0 
& \mbox{ in }\Omega\times\R^d,
\\[0.2 cm]
&\gamma_+\varphi(x,v) 
 = \gamma_-\varphi(x,\mathcal{R}_x(v)) 
& \mbox{ on }\Sigma_+,
\\[0.2 cm]
&\varphi(x,0)
 = \psi(x) 
& \mbox{ in } \Omega,
\end{aligned}\right.
\end{equation}
where we impose the initial condition on the hypersurface $\{v=0\}$. Note that $\phi^\eps(x,v) = \varphi(x,\eps v)$ will be a solution to the following auxiliary problem.
\begin{equation}\label{eq:Aux:auxiliary-problem-eps}
\left\{
\begin{aligned}
&\eps v \cdot \nabla_x \phi^\eps
- v \cdot \nabla_v \phi^\eps 
 = 0 
& \mbox{ in }\Omega\times\R^d,
\\[0.2 cm]
&\gamma_+\phi^\eps (x,v)
 = \gamma_-\phi^\eps(x,\mathcal{R}_x(v)) 
& \mbox{ on }\Sigma_+,
\\[0.2 cm]
&\phi^\eps(x,0)
 = \psi(x) 
& \mbox{ in } \Omega.
\end{aligned}\right.
\end{equation}
The characteristic curves associated with the boundary value problem \eqref{eq:Aux:auxiliary-problem} solve the following system of ordinary differential equations.
\begin{equation}\label{eq:cc}
  \left\{
      \begin{array}{llll}
        &\dot{x}(s) = v(s) \hspace{15mm} &x(0)=x_0,  \\[0.2 cm]
        &\dot{v}(s) = -v(s) &v(0)=v_0,\\[0.2 cm]
        & \text{If } x(s)\in\partial\Omega \text{ then } v(s^+)= \mathcal{R}_{x(s)}(v(s^{-})).
      \end{array}
    \right.
\end{equation}
We denote by $\Psi_{x_0,v_0}(s) = (x(s), v(s))$ to be the flow associated with \eqref{eq:cc} in the phase space $\Omega\times\R^d$ starting at $(x_0,v_0)$. Suppose the base point of the flow is an arbitrary $(x_0,v_0)\in\Omega\times\R^d$. With the convention $s_0=0$, consider the sequence $\{s_i\}_{i\ge0}\subset [0,\infty)$ of forward exit times defined as
\begin{align}\label{eq:Aux:exit-times}
s_{i+1}(x_0,v_0)
:= 
\inf 
\Big\{
\ell \in [s_i,\infty) \mbox{ s.t. } x(s_i) + (\ell - s_i) v(s_i) \notin \Omega 
\Big\}.
\end{align}
Solving \eqref{eq:cc} for the velocity component of the flow, we get
\begin{equation}
  \left\{
      \begin{aligned}
        &v(s) = e^{-s} v_0 &\text{ for } s\in[0,s_1),\\
		&v(s_i^+) = \mathcal{R}_{x(s_i)} v(s_i^-),\\        
        &v(s) = e^{-(s-s_i)} v(s_i^+)  &\text{ for } s\in(s_i,s_{i+1}),
      \end{aligned}
    \right.
\end{equation}
which gives the particle trajectory, for $s\in(s_i,s_{i+1})$,
\begin{align*}
x(s) &= x_0 +  \int\limits_0^s v(\tau) \text{d}\tau
	    = x_0 + \underset{k=0}{\overset{i-1}{\sum}} \int\limits_{s_{k}}^{s_{k+1}} v(\tau) \text{d}\tau +  \int\limits_{s_{i}}^{s} v(\tau) \text{d}\tau  \\
	 &= x_0 + \underset{k=0}{\overset{i-1}{\sum}} \left(1-e^{-(s_{k+1}-s_k)}\right) v(s_k^+) +  \left(1-e^{-(s-s_k)}\right)v(s_i^+).
\end{align*}
Instead of considering an exponentially decreasing velocity $v(s)$ on an infinite interval $[0,\infty)$, we would like to consider particle trajectories with constant speed on a finite interval $[0,1)$. To that end, we notice that the reflection operator $\mathcal{R}$ is isometric, which means
\begin{align*}
v(s_i^+) &= \mathcal{R}_{x(s_i)}\big( v(s_i^-)\big)\\
		  &= \mathcal{R}_{x(s_i)} \big( e^{-(s_i-s_{i-1})} v(s_{i-1}^+) \big)\\
		  &= e^{-(s_i-s_{i-1})} \mathcal{R}_{x(s_i)} \circ \mathcal{R}_{x(s_{i-1})} \big( e^{-(s_{i-1}-s_{s-2})} v(s_{i-2}^+) \big)\\
		  &= e^{-(s_i-s_{i-2})} \mathcal{R}_{x(s_i)} \circ \mathcal{R}_{x(s_{i-1})} \circ \mathcal{R}_{x(s_{i-2})} \big( e^{-(s_{i-2}-s_{s-3})} v(s_{i-3}^+)\big) \\
		  &= e^{-(s_i-0)} \mathcal{R}_{x(s_i)} \circ \mathcal{R}_{x(s_{i-1})} \circ \dots \circ \mathcal{R}_{x(s_1)} \big( v_0 \big).
\end{align*}
We define the operator $R^i$ as
\begin{equation} \label{def:Ri}
  \left\{
      \begin{aligned}
        & R^0 = Id,\\
        & R^i= \mathcal{R}_{x(s_i)} \circ R^{i-1},
      \end{aligned}
    \right.
\end{equation}
and a new velocity $w(s) := e^s v(s)$ which then satisfies
\begin{equation}
  \left\{
      \begin{aligned}
        &w(s) = v_0 &\text{ for } s\in(0,s_1),\\
		& w(s_i) = R^i v_0,\\        
        &w(s) = R^i w(s_i)  &\text{ for } s\in[s_i,s_{i+1}).
      \end{aligned}
    \right.
\end{equation}
It is easy to check that for any $s$, $|w(s)| = |v_0|$. The trajectory $x(s)$ can be written, with the new velocity variable $w$ as
\begin{align*}
x(s) &= x_0 + \int\limits_0^s e^{-\tau} w(\tau) \text{d}\tau\\
	 &= x_0 + \underset{k=0}{\overset{i-1}{\sum}} \left(e^{-s_k}-e^{-s_{k+1}}\right) w(s_k) + \left(e^{-s}-e^{-s_i}\right)w(s_i)
\end{align*}
and finally, we introduce a new parametrisation $\tau = 1 - e^{-s} \in [0,1)$ and the corresponding reflection times $\tau_i = 1-e^{-s_i}$ with which we have, for any $\tau \in [\tau_i,\tau_{i+1})$,
\begin{equation} \label{eq:etaexplicit}
  \left\{
      \begin{aligned}
        & x(\tau) = x_0 + \underset{k=0}{\overset{i-1}{\sum}} \left(\tau_{k+1}-\tau_{k}\right) w(\tau_k) + \left(\tau-\tau_i\right)w(\tau_i),\\
        & w(\tau) = w(\tau_i) = R^i w_0.
      \end{aligned}
    \right.
\end{equation}
We notice that the particle trajectory $x(\tau)$ together with the velocity profile $w(\tau)$ in \eqref{eq:etaexplicit} can be seen as the specular cycle associated with our Hamiltonian dynamics. Next, we record a couple of simple observations on the forward exit times $s_i$ and the grazing set $\Sigma_0$ associated with the specular cycle \eqref{eq:etaexplicit}.
\begin{Lem}\label{lem:Aux:exit-times-grazing-set}
Let $\Omega\subset\R^d$ be strictly convex. Then, we have
\begin{enumerate}
\item[(i)] For any $(x,v)\in\Omega\times\R^d$, the trajectory never passes through a grazing set $\Sigma_0$.
\item[(ii)] For any $(x,v)\in\Omega\times\R^d$, there exists a $N\in\mathbb{N}^*$ depending on $(x,v)$ such that the forward exit time $s_{N+1}(x,v)$ does not exist.
\end{enumerate}
\end{Lem}
The above result is proved in \cite[Chapter 1, Section 1.3, Lemma 1.3.17]{Safarov_1996}, an excellent book of Safarov and Vassiliev, where geodesic billiards on manifolds are extensively studied.

%%%%%%%%%%%%%%%%%%%%%%%%%%%%%%%
\subsection{Solution to the auxiliary problem and rescaling}\label{ssec:soln-Aux}
%%%%%%%%%%%%%%%%%%%%%%%%%%%%%%%
Next, we shall define a function on the phase space.
\begin{Defi}[End-point function]\label{def:Aux:end-point-function}
The end-point function $\eta:\big(\bar\Omega\times\R^d\big)\setminus\Sigma_0\to\bar\Omega$ is defined such that for every $(x_0,v_0)\in\bar\Omega\times\R^d\setminus\Sigma_0$, 
\begin{align*}
\eta(x_0,v_0) = x(\tau=1),
\end{align*}
where the particle trajectory is given in \eqref{eq:etaexplicit}.
\end{Defi}
Using the end-point function $\eta(x,v)$, we have a solution to the auxiliary problem \eqref{eq:Aux:auxiliary-problem} for any 
\begin{align}\label{eq:Aux:mathfrak-D}
\psi \in \mathfrak{D}:= \left\{ \psi\in C^\infty(\overline{\Omega})\mbox{ such that }\nabla\psi\cdot n(x) = 0 \mbox{ for }x\in \partial\Omega\right\},
\end{align}
which can be explicitly written as
\begin{align*}
\varphi(x,v) = \psi \left( \eta(x,v) \right).
\end{align*}
Hence we deduce a solution to the auxiliary problem \eqref{eq:Aux:auxiliary-problem-eps} for any $\psi\in\mathfrak{D}$ and for any $\eps>0$,
\begin{align}\label{eq:Aux:solution-phi-eps}
\phi^\eps (x,v) = \psi\left(\eta(x,\eps v)\right).
\end{align}
Indeed, the end-point function ensures not only that $\phi^\eps$ is constant along the specular cycles, which in turns implies that the first two equations of \eqref{eq:Aux:auxiliary-problem-eps} are satisfied, but also that $\phi^\eps(x,0)=\psi\big(\eta(x,0)\big) = \psi(x)$.\\
For $\phi^\eps$ to be a test function in the weak formulation \eqref{eq:soln-VFP:weak-soln-definition} of Vlasov-Fokker-Planck equation, we need to add a dependency in time. Hence taking $\psi(t,x)\in\mathfrak{D}$ for all $t\in[0,T]$, we have
\begin{align*}
\phi^\eps(t,x,v)= \psi(t,\eta(x,\eps v)).
\end{align*}
Finally, to conclude this section about the auxiliary problem, let us determine the limit of the family $\phi^\eps (t,x,v)$ as $\eps$ goes to $0$. By the definition of $\eta(x,v)$, for any $(x,v)\in \Omega\times\R^d$, there exists $\eps$ small enough, namely $\eps < \text{dist}(x,\partial\Omega)/ \vert v\vert$, such that $\eta(x,\eps v)=x+\eps v$. Therefore 
\begin{align*}
\underset{\eps \rightarrow 0}{\lim}\, \psi \left( t, \eta(x,\eps v) \right) 
= \underset{\eps \rightarrow 0}{\lim}\, \psi \left( t, x+\eps v \right)
= \psi(t,x) \quad \forall (t,x,v)\in[0,T]\times\Omega\times\R^d.
\end{align*}

%%%%%%%%%%%%%%%%%%%%%%%%%%%%
\section{Derivation of the macroscopic model}\label{sec:macro}
%%%%%%%%%%%%%%%%%%%%%%%%%%%%

We now return to the proof of Theorem \ref{thm:main:statement}. Consider $f^\eps(t,x,v)$ a weak solution of \eqref{eq:main:vfp-eps}-\eqref{eq:main:vfpin-eps}-\eqref{eq:main:vfp-specular-eps} in the sense of Definition \ref{def:soln-VFP:weak-solution}. For any $\phi^\eps$ satisfying \eqref{eq:soln-VFP:weak-soln-test-function} we have
\begin{equation}\label{eq:macro:weak-form}
\begin{aligned}
\iiint\limits_{(0,T)\times\Omega\times\R^d}
f^\eps(t,x,v) 
\Big( \eps^2 \partial_t \phi^\eps + \eps v \cdot \nabla_x \phi^\eps & - v \cdot \nabla_v \phi^\eps + \Delta_v \phi^\eps \Big)
\, {\rm d}v\, {\rm d}x\, {\rm d}t
\\
& + \eps^2 \iint\limits_{\Omega \times\R^d}
f^{in}(x,v) \phi^\eps(0,x,v)
\, {\rm d}v\, {\rm d}x
= 0.
\end{aligned}
\end{equation}
In particular, for $\phi^\eps(t,x,v) = \psi\left( t,\eta(x,\eps v) \right)$, where $\psi(t,x)\in\mathfrak{D}\, \forall t\in[0,T]$, we have
\begin{align*}
  \left\{
      \begin{aligned}
        & \eps v \cdot \nabla_x \left[ \psi \left( t,\eta(x,\eps v) \right) \right] - v \cdot \nabla_v \left[ \psi \left( t,\eta(x,\eps v) \right) \right] = 0\\[0.2 cm]
        & \Delta_v \left[ \psi \left( t,\eta(x,\eps v) \right) \right] = \eps^2 \Delta_v \left[ \psi \left( t,\eta(x,\cdot) \right) \right] (\eps v).
      \end{aligned}
    \right.
\end{align*}
Hence, \eqref{eq:macro:weak-form} becomes
\begin{align}\label{eq:macro:weak-formulation-2}
\iiint\limits_{(0,T)\times\Omega\times\R^d}
f^\eps \left( \partial_t \psi + \Delta_v \left[ \psi \left( t,\eta(x,\cdot) \right) \right] (\eps v) \right)
\, {\rm d}v\, {\rm d}x\, {\rm d}t
+ \iint\limits_{\Omega\times\R^d} 
f^{in}(x,v) \psi \left( 0,\eta(x,\eps v) \right) 
\, {\rm d}v\, {\rm d}x 
= 0.
\end{align}
Since $f^\eps$ converges weakly* in $\mathrm L^\infty(0,T;\mathrm L^2(\mathcal{M}^{-1}(v){\rm d}x{\rm d}v))$ (Proposition \ref{prop:soln-VFP:weak-limit}), in order to take the limit as $\eps$ goes to $0$ we need to show that $\Delta_v \left[ \psi \left( t,\eta(x,\cdot) \right) \right] (\eps v)$ converges strongly in $\mathrm L^2(\mathcal{M}(v){\rm d}x{\rm d}v)$. To that end, we write
\begin{align}
\Delta_v & \left[ \psi \left( t,\eta(x,\cdot) \right) \right] (\eps v) 
= \nabla_v \cdot \nabla_v \left( \psi \left( t,\eta(x,\cdot) \right) \right) (\eps v )
\nonumber \\
&
= \sum_{i=1}^d\sum_{k=1}^d
\frac{\partial^2 \eta_k}{\partial v_i^2} (x,\eps v) \frac{\partial\psi}{\partial\eta_k} (t,\eta(x,\eps v))
\nonumber \\
& \quad
+ \eps^2 \sum_{i=1}^d\sum_{k=1}^d\sum_{l=1}^d 
\frac{\partial \eta_k}{\partial v_i}(x,\eps v)\frac{\partial^2\psi}{\partial\eta_k\partial\eta_l}(t,\eta(x,\eps v))\frac{\partial \eta_l}{\partial v_i}(x,\eps v)
\nonumber \\
&
= \Delta_v \eta(x,\eps v) \cdot \nabla_x \psi \left( t,\eta(x,\eps v) \right) + \Tr \left( \nabla_v\eta(x,\eps v) {}^\top\!\nabla_v\eta(x,\eps v) H_x\psi \left( t,\eta(x,\eps v) \right) \right), \label{eq:devDeltapsieta}
\end{align}
where $H_x\psi$ denotes the Hessian matrix of $\psi$. For any $(x,v)\in\Omega\times\R^d$ we know that for $\eps$ small enough, i.e. $\eps < \text{dist}(x,\partial\Omega)/\vert v\vert$, $\eta(x,\eps v) = x+\eps v$, which means $\nabla_v \eta(x,\eps v) = {\rm Id}$ and $\Delta_v \eta(x,\eps v)= 0$ so that, using the computation above, for such $\eps$ we have
\begin{align*}
\Delta_v \left[ \psi \left( t,\eta(x,\cdot) \right)\right] (\eps v) 
= \Tr \left(H_x\psi \left(t,x+\eps v\right)\right) 
= \Delta_x \psi \left(t,x + \eps v\right).
\end{align*}
Since $\psi$ is smooth, this yields a point-wise convergence
\begin{align}\label{eq:macro:pointwise}
\Delta_v \left[ \psi \left( t,\eta(x,\cdot) \right) \right] (\eps v)  \rightarrow \Delta_x\psi(t,x) \quad \mbox{ a.e. on } [0,T]\times\Omega.
\end{align}
This convergence holds up to the boundary, indeed for any $x\in \partial\Omega$ and $v\in\mathbb{R}^d$ we see, by the definition of the end-point function, that for some $\eps$ small enough
\begin{equation*} 
\nabla_v \eta (x, \eps v) = \left\{ \begin{aligned}  &{\rm Id}  &\quad\mbox{ if } v\cdot n(x) <0 \\
																			  &{\rm Id} - 2 n(x) \otimes n(x) &\quad\mbox{ if } v\cdot n(x) >0 \end{aligned} \right.
\end{equation*}
which yields, in turn, that $\Delta \eta (x,\eps v) = 2 n(x)\delta_{v\cdot n(x)=0}$. Hence, for any $\psi \in \mathfrak{D}$  and $(x,v)\in \partial \Omega\times \mathbb{R}^d$ we have
\begin{align*}
\Delta_v \left[ \psi \left( t,\eta(x,\cdot) \right) \right] (\eps v) \rightarrow 2 \nabla \psi(x)\cdot n(x) \delta_{v\cdot n(x)=0} + \Delta \psi(x) = \Delta\psi(x).
\end{align*}
Finally, we have the following.
\begin{Lem} \label{lem:eta}
If $\Omega$ is a unit ball in $\R^d$ and $\eta$ is defined as in Definition \ref{def:Aux:end-point-function} on $\Omega$ then we have
\begin{equation} \label{eq:regeta}
\underset{r>0}{\sup} \Big( \Delta_v \left[ \psi \left( t,\eta(x,\cdot) \right) \right] ( r v) \Big) \in \mathrm L^\infty\big( (0,T) ;\mathrm L^2( \Omega\times\mathbb{S}^{d-1}) \big)
\end{equation}
for any $\psi \in \mathfrak{D}_T$, where
\begin{align*}
\mathfrak{D}_T :=  
\big\{
\psi \in \mathcal{C}^\infty([0,T)\times \overline{\Omega}) \mbox{ s.t. } \psi(T,\cdot)=0\, \mbox{ and }\,  n(x)\cdot\nabla_x\psi(t,x) = 0
\text{ on } (0,T)\times\partial\Omega 
\big\}.
\end{align*}
\end{Lem}
To prove this lemma we study the regularity of the end-point function $\eta(x,v)$, which is rather technical and will be the subject of the appendix of this paper, based on results from \cite{Cesbron_2016}. Nevertheless, this allows us to use the Lebesgue's dominated convergence theorem in $\mathrm L^2(\mathcal{M}(v){\rm d}x{\rm d}v)$ and pass to the limit in the weak formulation \eqref{eq:macro:weak-formulation-2}  as $\eps$ goes to $0$ to get
\begin{align}\label{eq:macro:weakformlimit-bis}
\iint\limits_{(0,T)\times\Omega}
\rho(t,x) \Big( \partial_t \psi(t,x) + \Delta_x \psi(t,x) \Big)
\, {\rm d}x\, {\rm d}t 
+ \int\limits_\Omega 
\rho^{in}(x)\psi(0,x)
\, {\rm d}x 
= 0,
\end{align}
which holds for any $\psi \in \mathfrak{D}_T$. To conclude the proof of Theorem \ref{thm:main:statement}, we need to show that the solution $\rho$ of \eqref{eq:macro:weakformlimit-bis} is a weak solution to the diffusion equation \eqref{eq:main:diffusion-equation}-\eqref{eq:main:diffusion-initial}-\eqref{eq:main:diffusion-neumann}, which is the objective of the following proposition.

\begin{Prop}
If $\rho$ satisfies, for every $\psi\in \mathfrak{D}_T$,
\begin{equation}\label{eq:macro:weakformlimit} 
\iint\limits_{(0,T)\times\Omega}
\rho(t,x) \Big( \partial_t \psi + \Delta_x \psi \Big)(t,x)
\, {\rm d}x\, {\rm d}t 
+ \int\limits_\Omega 
\rho^{in}(x)\psi(0,x)
\, {\rm d}x 
= 0,
\end{equation}
then $\rho$ is the unique solution of the heat equation with homogeneous Neumann boundary condition, i.e., for any $\psi \in \mathrm L^2\big( 0,T; \mathrm H^1(\Omega)\big)$,
\begin{equation} \label{eq:macro:weakheat}
\int\limits_0^T
\left\langle
\partial_t \rho,\, \psi
\right\rangle_{V', V}
\, {\rm d}t 
+ \iint\limits_{(0,T)\times\Omega}
\nabla_x \rho(t,x)\cdot \nabla_x \psi(t,x) 
\, {\rm d}x\, {\rm d}t 
= 0,
\end{equation}
where $V=\mathrm H^1(\Omega)$ and $V'$ is its topological dual.
\end{Prop}

\begin{proof}
This proof consists in showing that the solution $\rho$ of \eqref{eq:macro:weakformlimit} is regular enough for \eqref{eq:macro:weakheat} to make sense. Once this is established, a classical density argument will conclude the proof of the proposition, and therefore the proof of Theorem \ref{thm:main:statement}, by showing that \eqref{eq:macro:weakheat} holds for any $\psi$ in $\mathrm L^2\big( 0,T;\mathrm H^1(\Omega)\big)$. \\
For any $u \in C^\infty([0,T);C^\infty_c(\Omega))$, we consider the unique solution to the boundary-value problem
\begin{equation} \label{eq:macro:Laplace} 
\left\{ \begin{aligned}  &\Delta_x \psi (t,x) = \frac{\partial u}{\partial x_i}(t,x)  &\mbox{ in } (0,T)\times\Omega,\\
									 &\nabla \psi(t,x) \cdot n(x) = 0 &\mbox{ on } (0,T)\times\partial \Omega,\\
									 &\int\limits_\Omega \psi (t,x) \, {\rm d}x  = 0,
\end{aligned} \right. 
\end{equation}
for any $i \in \{ 1,\cdots, d\}$. Notice that the time variable $t$ in \eqref{eq:macro:Laplace} plays the role of a parameter. It is well known that the solution $\psi$ to \eqref{eq:macro:Laplace} will be in $\mathfrak{D}_T$. To derive the energy estimate, multiply \eqref{eq:macro:Laplace} by $\psi$ and integrate over $\Omega$ yielding
\begin{align}\label{eq:macro:Laplace-energy}
\int\limits_\Omega \psi(t,x) \Delta_x \psi(t,x) \, {\rm d}x = \int\limits_\Omega \psi(t,x) \frac{\partial u}{\partial x_i}(t,x) \, {\rm d}x
\quad
\forall t\in[0,T].
\end{align} 
On the left-hand side, the homogeneous Neumann condition in \eqref{eq:macro:Laplace} yields
\begin{align*}
\left| \int\limits_\Omega \psi(t,x) \Delta_x \psi(t,x) \, {\rm d}x  \right| = \|\nabla \psi(t,\cdot) \|_{\mathrm L^2(\Omega)}^2
\quad
\forall t\in[0,T].
\end{align*}
On the right hand-side of \eqref{eq:macro:Laplace-energy}, since $u$ is compactly supported in $\Omega$ we can write
\begin{align*}
\left| \int\limits_\Omega \psi(t,x) \frac{\partial u}{\partial x_i}(t,x) \, {\rm d}x \right| = \left| \int\limits_\Omega u(t,x) \frac{\partial \psi}{\partial x_i}(t,x) \, {\rm d}x \right| \leq \|u(t,\cdot)\|_{\mathrm L^2(\Omega)} \|\nabla \psi(t,\cdot)\|_{\mathrm L^2(\Omega)}
\quad
\forall t\in[0,T].
\end{align*}
Together with the Poincar\'e inequality, this computation shows that $\|\psi(t,\cdot)\|_{\mathrm L^2(\Omega)} \leq \|u(t,\cdot)\|_{\mathrm L^2(\Omega)}$ for all $t\in[0,T]$. Taking the thus constructed $\psi(t,x)$ as the test function in the formulation \eqref{eq:macro:weakformlimit}, we get
\begin{align*}
\left| \, \, \,
\iint\limits_{(0,T)\times\Omega}
\rho(t,x) \frac{\partial u}{\partial x_i}
\, {\rm d}x\, {\rm d}t
\right|
\le 
\left|
\int\limits_\Omega 
\rho^{in}(x) \psi(0,x) 
\,{\rm d}x
\right|
+
\left| \,\,\,
\iint\limits_{(0,T)\times\Omega}
\rho(t,x)
\partial_t \psi(t,x)
\, {\rm d}x\, {\rm d}t
\right|,
\end{align*}
which, in particular, implies that for any $u\in\mathcal{D}(\Omega)$, considering $\psi$ that doesn't depend on $t$ and with a constant $C=\Vert\rho^{in}\Vert_{\mathrm L^2(\Omega)}$, we arrive at the following control
\begin{align*}
\left|
\int\limits_0^T
\left\langle
\frac{\partial \rho}{\partial x_i}, u
\right\rangle_{\mathcal{D}'(\Omega), \mathcal{D}(\Omega)}
\, {\rm d}t
\right|
\le 
C \|u\|_{\mathrm L^2(\Omega)}.
\end{align*}
The above observation implies that
\begin{align*}
\rho\in \mathrm L^2(0,T;\mathrm H^1(\Omega)).
\end{align*}
It is a classical matter to show that $\mathfrak{D}_T$ is dense in $\mathrm L^2(0,T;\mathrm H^1(\Omega))$. Using the above regularity of $\rho$ in \eqref{eq:macro:weakformlimit} and taking $\psi\in \mathrm L^2(0,T;\mathrm H^1(\Omega))$ would yield the following regularity on the time derivative
\begin{align*}
\partial_t\rho\in \mathrm L^2(0,T;V'),
\end{align*}
where $V'$ is the topological dual of $V=\mathrm H^1(\Omega)$. Thus, we have proved that the limit local density $\rho(t,x)$ is the unique solution of the weak formulation \eqref{eq:macro:weakheat}.
\end{proof}

\begin{Rem}\label{rem:unit-ball-eta}
Note that the result in Lemma \ref{lem:eta} is given for a particular choice of the spatial domain -- a ball in $\R^d$. We are unable so far to prove a similar regularity result in more general strictly convex domains.
\end{Rem}

%%%%%%%%%%%%%%%%%%%%%%
\section{Appendix}
%%%%%%%%%%%%%%%%%%%%%%

In this section, we let the spatial domain $\Omega$ be the unit ball in $\R^d$. We consider the trajectories in $\Omega$ described by \eqref{eq:etaexplicit} and the associated end-point function $\eta(x,v)$. The purpose of this section is to prove Lemma \ref{lem:eta}.
\begin{proof}[Proof of Lemma \ref{lem:eta}]
We first note that a trajectory in $\Omega$ is necessarily included in a plane of dimension 2. Indeed, by definition of the specular reflection, when the trajectory hits the boundary, the reflected velocity is a linear combination of the initial velocity and the normal vector: $Rv = v - 2 \left( n(x+sv) \cdot v\right) n(x+sv)$ where $n(x+sv)=x+sv$, since $\Omega$ is a unit ball. Since the normal vector belongs to the plane generated by $x$ and $v$, we see that the reflected velocity also belongs to that same plane, and the same goes for every reflected velocity along this trajectory. As a consequence, we restrict the study of the regularity of the end-point function $\eta(x,v)$ in a ball to the case of a disk in dimension $d=2$.\\
Let us recall that
\begin{align*}
\Delta_v & \left[ \psi \left( t,\eta(x,v) \right) \right] 
= \Delta_v \eta(x, v) \cdot \nabla_x \psi \left( t,\eta(x, v) \right) + \Tr \left( \nabla_v\eta(x, v) {}^\top\!\nabla_v\eta(x, v) H_x\psi \left( t,\eta(x, v) \right) \right).
\end{align*}
From \cite[Appendix A.2]{Cesbron_2016}, we recall that $\na_v \eta(x,v)$ is uniformly bounded in $x$ and $v$, so the second term in the above expression is immediately handled. For the first term, if we write $k$ the number of reflections and define $A$ the angle of reflection and $L$ the length between two reflections, illustrated in Figure \ref{fig:notation} and constant along a trajectory, 
\begin{figure}[h]
\centering
\includegraphics[width=11cm,height=8.5cm]{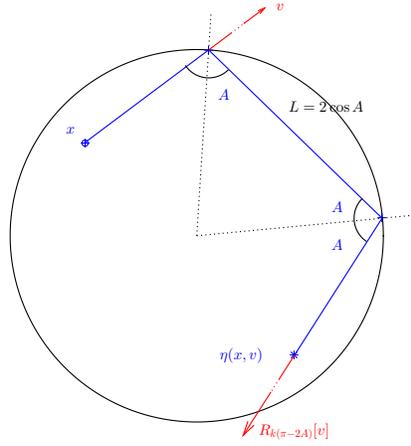}
\caption{Trajectory with 2 reflections in the circle}
\label{fig:notation}
\end{figure}
then from the expression of $D^2 \eta(x,v)$, again borrowed from \cite[Appendix A.2]{Cesbron_2016}, it is easy to see that the Laplacian of $\eta$ can be written as
\begin{align*}
\Delta \eta(x,v) =  \frac{1}{L^2} \lambda SR_{k(\pi-2A)} \frac{v}{|v|} + C
\end{align*}
where $\lambda=\lambda(x,v)$ and $C=C(x,v)$, both uniformly bounded in $x$ and $v$, $S$ is the symmetry matrix: $S=\begin{pmatrix} 0 & 1 ; -1 , 0 \end{pmatrix}$, and $R_{k(\pi-2A)}$ is the rotation matrix of angle $k(\pi-2A)$. \\
Moreover, when we start close to the grazing set, the trajectory stays close to the grazing set (because $A$ is a constant close to $\pi/2$), which means $R_{k(\pi-2A)} v/|v|$ stays close to $\tau(\eta(x,v))$, then tangent of $\Omega$ at $\eta(x,v)/|\eta(x,v)| \in \pa\Omega$. In fact it will be furthest from the tangent when $\eta(x,v)$ is on the boundary where we have
\begin{align*}
R_{k(\pi-2A)} \frac{v}{|v|} &= \big( \cos A \big)  n\big(\eta(x,v)\big) + \big( \sin A  \big) \tau\big(\eta(x,v)\big) \\
			&= \Big( \frac{1}{2} L \Big) n\big(\eta(x,v)\big) + \Big(1- \frac{L^2}{4}\Big)^{1/2} \tau \big(\eta(x,v)\big)
\end{align*}
so that 
\begin{align*}
SR_{k(\pi-2A)} \frac{v}{|v|}= n\big(\eta(x,v)\big)  + O(L)
\end{align*}
where $n(\eta(x,v))$ is the outward normal at $\eta(x,v)/|\eta(x,v)|\in\pa\Omega$. \\
Furthermore, if we consider $\psi\in\mathfrak{D}_T$ then on the boundary, $\na \psi(x,v) \cdot n(x) =0$ hence, by the regularity of $\psi$, when $\eta(x,v)$ is close the boundary we have
\begin{align*}
\na \psi\big(\eta(x,v)\big)\Big) = \tau\big( \eta(x,v) \big) + O\Big(\text{dist}\big( \eta(x,v) , \pa\Omega \big)\Big).
\end{align*}
We can bound the distance between $\eta(x,v)$ and the boundary in terms of $L$ because we are in a circle so the $\eta(x,v)$ is furthest from the boundary when it is in the middle between two reflections and the Pythagorean theorem tells us in that case
\begin{align*}
\Big( 1 - \text{dist}\big(\eta(x,v), \pa\Omega \big) \Big)^2 + \Big( \frac{L}{2} \Big)^2 = 1
\end{align*}
so that we have all along the trajectory
\begin{align*}
\text{dist}\big(\eta(x,v), \pa\Omega\big) = 1 - \sqrt{1- \frac{L^2}{4}}  = \frac{L^2}{4} + o(L^2).
\end{align*}
All together, this yields
\begin{align*}
\Delta \eta \cdot \na \psi\big(\eta(x,v)\big) &= \frac{\lambda}{L} SR_{k(\pi-2A)}\frac{v}{|v|} \cdot \na \psi\big(\eta(x,v)\big) + O(1)\\
																	&= O\Big(\frac{1}{L} \Big).
\end{align*}
To investigate the integrability of $1/L$ we express $L$ in terms of $x$ an $v$. Since $L=2\cos A$ where $\cos A = (x+tv)\cdot v$ for some $t$ such that $|x+tv|^2 =1$ one can deduce that
\begin{align*}
L = 2 \sqrt{ (x\cdot \underline{v})^2 + 1-|x|^2}
\end{align*}
where $\underline{v} = v/|v|$. Note in the fact that $L$ does not depend on the norm of $v$ from which we see that we can take the supremum over the norm of $v$ and it won't impact the integrability in $\mathrm L^2(M(v) {\rm d}x {\rm d}v)$ since $M$ is radial and normalized. \\
To conclude the proof of Lemma \ref{lem:eta}, we use a polar change of variables to write for some $p>0$,
\begin{align*}
\underset{\Omega\times\mathbb{S}^1}{\iint}& \Big(\frac{2}{L}\Big)^p {\rm d} x {\rm d} \underline{v} =  \underset{\Omega\times\mathbb{S}^1}{\iint} \frac{1}{\Big( (x\cdot \underline{v})^2 + (1-|x|^2) \Big)^{p/2}} {\rm d} x {\rm d} \underline{v} \\
&= \int_0^{1} \int_0^{2\pi} \int_0^{2\pi} \frac{\rho_x}{\big( 1- \rho_x^2 \sin^2(\theta_v-\theta_x) \big)^{p/2}} \, {\rm d} \rho_x \, {\rm d}\theta_x \, {\rm d}\theta_v \\
&= \int_0^{1} \int_0^{2\pi} \int_0^{2\pi} \frac{\rho_x}{\big( 1- \rho_x| \sin(\theta_v-\theta_x)| \big)^{p/2} \big( 1+ \rho_x| \sin(\theta_v-\theta_x)| \big)^{p/2}} \, {\rm d}\rho_x \, {\rm d}\theta_x \, {\rm d}\theta_v \\
&\leq  C  \int_0^{1} \int_0^{2\pi} \int_0^{2\pi} \frac{\rho_x}{\big( 1- \rho_x| \sin(\theta_v-\theta_x)| \big)^{p/2}} \, {\rm d}\rho_x \, {\rm d}\theta_x \, {\rm d}\theta_v \\
&\leq   2 \pi C  \int_0^{1} \int_0^{2\pi} \frac{\rho_x}{\big( 1- \rho_x| \sin\alpha| \big)^{p/2}} \, {\rm d}\rho_x \, {\rm d}\alpha \\
&\leq  \tilde{C}  \int_0^1 \int_0^{\sqrt{1-x_2^2}} \frac{1}{\big( 1 - x_2)^{p/2}} {\rm d}x_1 {\rm d}x_2  \\
&\leq  \tilde{C} \int_0^1 \frac{1}{(1-x_2)^{p/2-1/2}} {\rm d}x_2 
\end{align*}
hence $1/L$ will be in $L^p_{F(v)}(\Omega\times\R^2)$ if $p < 3$. \\
As a remark, note however that if we took the supremum in $v$ instead $|v|$ then $1/L$ would be equivalent to $1/\sqrt{1-|x|^2}$ which is in $L^{2-\delta}(\Omega)$ for all $\delta>0$ but not for $\delta=0$. 
\end{proof}

%\begin{repLem}{lem:eta}
%If $\Omega$ is a unit ball in $\R^d$ and $\eta$ is defined as in Definition \ref{def:Aux:end-point-function} on $\Omega$ then for any $\psi \in \mathfrak{D}_T$ we have
%\begin{equation} \tag{\ref{eq:regeta}}
%\underset{r>0}{\sup} \Big( \Delta_v \left[ \psi \left( t,\eta(x,\cdot) \right) \right] ( r \underline{v}) \Big) \in L^\infty\big( (0,T) ;L^2( \Omega\times\mathbb{S}^{d-1}) \big).
%\end{equation}
%\end{repLem}

\bibliography{biblio}

\end{document}